\documentclass{amsart}
\usepackage{amsmath, amscd, amssymb, amsthm}
\usepackage{bbm}
\usepackage{latexsym}
\usepackage{amsfonts}
\usepackage{graphicx}
\usepackage[all,cmtip]{xy}
\usepackage[colorlinks,linkcolor=blue,breaklinks=true,urlcolor=blue,citecolor=blue,anchorcolor=blue,pagebackref]{hyperref}%
\usepackage{geometry}
\newtheorem{theorem}{Theorem}
\newtheorem{lemma}{Lemma}
\newtheorem{corollary}[theorem]{Corollary}

\newtheorem{proposition}{Proposition}

\newtheorem{conjecture}{Conjecture}

\newcommand{\R}{{\mathbb R}}

\newcommand{\C}{{\mathbb C}}

\renewcommand*\backref[1]{}
\renewcommand*\backrefalt[4]{ \ifcase #1 \or (cited on page #2) \else (cited on pages #2) \fi}

\newcommand{\be}{\begin{equation}}
\newcommand{\ee}{\end{equation}}
\newcommand{\bea}{\begin{eqnarray}}
\newcommand{\eea}{\end{eqnarray}}

\newcommand{\tr}{\mathrm{tr}}
\newcommand{\im}{\mathrm{Im}}

\DeclareMathOperator{\End}{End}
\DeclareMathOperator{\Hol}{Hol}

\def\XXint#1#2#3{{\setbox0=\hbox{$#1{#2#3}{\int}$ }
\vcenter{\hbox{$#2#3$ }}\kern-.6\wd0}}

\begin{document}

\title[Hermitian connections with parallel torsion]
{Hermitian connections with parallel torsion and constant holomorphic sectional curvature}

\author{Shuwen Chen}
\address{School of Mathematical Sciences, Chongqing Normal University, Chongqing 401331, China}
\email{3153017458@qq.com}

\author{Fangyang Zheng}
\address{School of Mathematical Sciences, Chongqing Normal University, Chongqing 401331, China}
\email{20190045@cqnu.edu.cn; franciszheng@yahoo.com}
\thanks{F. Zheng is the corresponding author. He is partially supported by the National Natural Science Foundation of China under grants No.~12471039 and No.~12141101, and by the 111 Project D21024.}


\begin{abstract}
A long-standing conjecture in non-K\"ahler geometry states that a compact Hermitian manifold with constant Chern holomorphic sectional curvature should be K\"ahler when the constant is nonzero and Chern flat when the constant is zero. In this article, we study the corresponding problem for a Hermitian connection $\nabla$ with $\nabla T=0$. We first show that if the $\nabla$-holomorphic sectional curvature is constant, then the $(1,1)$-part of the curvature is $\nabla$-parallel. Our main result states that a nonzero constant forces $T=0$; consequently, the metric is K\"ahler and locally a complex space form. The proof is pointwise and requires neither compactness nor completeness. We also obtain applications to the Chern and Bismut connections.
\end{abstract}

\maketitle

\markleft{Shuwen Chen and Fangyang Zheng}
\markright{Hermitian connections with parallel torsion and constant holomorphic sectional curvature}

\tableofcontents

\section{Introduction}

Let $(M^n,g)$ be a Hermitian manifold. We denote by $\nabla^c$ and
$R^c$ the Chern connection and its curvature tensor, respectively.
For $0\neq X\in T^{1,0}M$, its Chern holomorphic sectional curvature is $H^c(X)=R^c_{X\bar X X\bar X}/|X|^4$. When $g$ is K\"ahler, the K\"ahler curvature symmetries imply that the values of $H^c$ determine the entire curvature tensor. In particular, complete K\"ahler manifolds with constant holomorphic sectional curvature are the classical \emph{complex space forms}, whose simply connected models are $\mathbb{CP}^n$, $\C^n$, and $\mathbb{CH}^n$, equipped with the standard metrics up to a constant rescaling.

For a general Hermitian metric, the Chern curvature tensor does not obey all the K\"ahler symmetries. Thus the values of the holomorphic sectional curvature determine only the symmetrization of $R^c$, instead of the entire curvature tensor. This loss of
curvature symmetry is one of the main difficulties in the study of Hermitian space forms. The following is a long-standing conjecture in non-K\"ahler geometry.

\begin{conjecture}[Constant holomorphic sectional curvature conjecture]\label{conj:Chern-HSC}
Let $(M^n,g)$ be a compact Hermitian manifold whose Chern holomorphic sectional curvature is a constant $c$. If $c\neq0$, then $g$ is K\"ahler; if $c=0$, then $g$ is Chern flat.
\end{conjecture}

By the classical theorem of Boothby \cite{Boothby}, compact Chern-flat manifolds are compact quotients of complex Lie groups equipped with left-invariant Hermitian metrics. Such metrics need not be K\"ahler in complex dimension three or higher.

The conjecture is known in complex dimension two. Balas and
Balas--Gauduchon proved the cases $c\leq0$, while Apostolov,
Davidov, and Muskarov settled the remaining positive case
\cite{Balas,BG,ADM}. In complex dimension at least three, the
general conjecture remains open, although several important special cases have been established. Davidov, Grantcharov, and Muskarov studied the problem on twistor spaces \cite{DGM}. Chen, Chen, and Nie proved the locally conformally K\"ahler case when $c\leq0$, and Huang and Wan subsequently treated the positive case
\cite{CCN,HuangWan}. Tang confirmed the conjecture for Chern
K\"ahler-like metrics \cite{Tang}. Results for balanced threefolds under additional assumptions were obtained by Zhou and Zheng and by Ma and Nie \cite{ZhouZ,MaNie}. Chen and Li proved the conjecture for compact balanced threefolds when $c\leq0$ \cite{ChenLi}. In a recent preprint, Qin and Tian
\cite{QinTian} proved the positive case for compact balanced threefolds. They also showed that every Hermitian metric of constant nonpositive Chern holomorphic sectional curvature on a compact complex manifold in Fujiki's class $\mathcal C$ is K\"ahler. Li and Zheng proved the conjecture for complex nilmanifolds \cite{LZ}, while Rao and Zheng established it for Bismut K\"ahler-like metrics \cite{RZ}.

An important class arising naturally in this problem is that of
\emph{Bismut torsion-parallel} metrics, abbreviated as BTP metrics from now on. Let $\nabla^b$ and $T^b$ denote the Bismut connection \cite{Bismut} and its torsion, respectively. By definition, a Hermitian metric is BTP if $\nabla^bT^b=0$. This class contains several familiar special Hermitian geometries. For instance, every Bismut-flat metric is BTP \cite{WYZ}. Moreover, every Bismut K\"ahler-like metric is BTP \cite{ZhaoZBTP}, while every Vaisman metric is BTP \cite{AndradaVillacampa,ZhaoZBTP}. Here a metric is called \emph{Bismut K\"ahler-like} if the curvature of $\nabla^b$ satisfies all the K\"ahler curvature symmetries, while a \emph{Vaisman metric} is a locally conformally K\"ahler metric whose Lee form is parallel with respect to the Levi--Civita connection. In particular, the BTP condition includes a number of well-studied non-K\"ahler examples.

The constant holomorphic sectional curvature problem has been studied extensively within the BTP class. We proved Conjecture~\ref{conj:Chern-HSC} for non-balanced BTP metrics in all dimensions and, using the classification of Zhao and Zheng, for balanced BTP threefolds \cite{ChenZBTP}. Wang and Zheng subsequently established the conjecture for balanced BTP fourfolds and then for balanced BTP manifolds in all dimensions \cite{WangZFourfold,WangZAll}. An independent proof of the nonzero-constant case for balanced BTP manifolds was obtained by H. Wang \cite{WangHaohao}.

In Conjecture \ref{conj:Chern-HSC}, one can replace the Chern connection by Levi-Civita or Bismut connection, or more generally, by any canonical metric connection
$$ D^t_s = (1-s)D^t + s \nabla^r, \ \ \   \ \ \ D^t=\frac{1+t}{2}\nabla^c + \frac{1-t}{2}\nabla^b,$$
where $t,s\in {\mathbb R}$, $\nabla^r$, $\nabla^c$, $\nabla^b$ are respectively the Levi-Civita, Chern, or Bismut connection. The one parameter family of connections $D^t$ is called the Gauduchon line. In \cite{CN}, Chen and Nie proposed the following:

\begin{conjecture}[Chen-Nie] \label{conj2}
Let $(M^n,g)$ be a compact Hermitian manifold. If for a given $(t,s)$, the holomorphic sectional curvature of $D^t_s$ is a non-zero constant, then $g$ must be K\"ahler.
\end{conjecture}

The situation for zero holomorphic sectional curvature is a little more complicated so we omit the discussion here. In \cite{CN}, they confirmed the conjecture in complex dimension 2.

One could also consider the space form problem for general Hermitian connections on $(M^n,g)$, which form an infinite dimensional affine space containing the Gauduchon line $D^t$. As observed in \cite{WYZ26}, {\em for any $n\geq 2$, there are compact Hermitian manifolds $(M^n,g)$  which admit a Hermitian connection $\nabla$ whose holomorphic sectional curvature is a nonzero constant, yet the metric $g$ is not K\"ahler.} So one cannot extend Conjecture \ref{conj2} to general Hermitian (or metric) connections. In this article, we restrict ourselves to a special type of Hermitian connections, namely, those with parallel torsion. This is a direct generalization to BTP manifolds. We want to know when can the holomorphic sectional curvature of such a connection be equal to a non-zero constant.

More specifically, let $\nabla$ be a \emph{Hermitian connection} on a Hermitian manifold $(M^n,g)$, that is, $\nabla g=0$ and $\nabla J=0$. We denote its torsion, curvature, and holomorphic sectional curvature by $T$, $R$, and $H$, respectively. Motivated by Conjectures~\ref{conj:Chern-HSC} and \ref{conj2}, we study the rigidity of such a connection under the assumptions $\nabla T=0$ and $H\equiv c$. Here and below, $R^{1,1}$ denotes the $(1,1)$-component of the $\End(T_\C M)$-valued curvature $2$-form with respect to its first two arguments. Our first result is a partial Ambrose--Singer statement.

\begin{theorem}\label{thm:partial-AS}
Let $(M^n,g)$ be a Hermitian manifold endowed with a Hermitian
connection $\nabla$ whose torsion $T$ is $\nabla$-parallel. If $H$ is globally constant, namely  $H\equiv c$, then
$\nabla(R^{1,1})=0$. In particular, if $R$ is of type $(1,1)$, then $\nabla R=0$, and
$\nabla$ is an Ambrose--Singer connection.
\end{theorem}

We remark that the constancy of $c$ in Theorem~\ref{thm:partial-AS} is important. If one only assumes that the holomorphic sectional curvature is pointwise constant, namely $H=f(p)$ for a smooth function $f$,  then the proof of Theorem~\ref{thm:partial-AS} no longer yields parallel curvature, as terms involving $df$ will appear. The main theorem is the following.

\begin{theorem}\label{thm:main}
Let $(M^n,g)$ be a Hermitian manifold endowed with a Hermitian
connection $\nabla$ whose torsion $T$ is $\nabla$-parallel. If $H\equiv c\neq0$, then $T=0$. Consequently, $g$ is K\"ahler and locally a complex space form.
\end{theorem}

The last sentence in the above theorem means that the universal cover of $(M^n,g)$ is holomorphically isometric to an open subset of a complex space form. We remark that Theorem~\ref{thm:main} requires neither compactness
nor completeness. On the other hand, the assumption $c\neq0$ cannot
be removed. Standard isosceles Hopf manifolds have parallel Bismut
torsion and satisfy $H^b\equiv0$, while their Bismut curvature need
not vanish in complex dimension at least three
\cite{ChenZStrominger}. Nevertheless, Theorem~\ref{thm:partial-AS} still shows that the $(1,1)$-part of the
curvature is parallel when $c=0$.

As applications, we obtain the following consequences.

\begin{corollary}\label{cor:Chern}
Let $(M^n,g)$ be a complete Hermitian manifold with $\nabla^cT^c=0$ and constant Chern holomorphic sectional curvature $H^c\equiv c$. If $c\neq0$, then $g$ is K\"ahler and locally a complex space form. If $c=0$, then $g$ is Chern flat.
\end{corollary}

The case $c=0$ in Corollary~\ref{cor:Chern} uses the structure
theorem for locally Chern homogeneous Hermitian manifolds due to
Ni--Zheng \cite{NiZhengCAS}. We next turn to the Bismut connection.

\begin{corollary}\label{cor:BAS}
Let $(M^n,g)$ be Bismut torsion-parallel and suppose that its Bismut holomorphic sectional curvature is constant. Then $\nabla^b$ is an Ambrose--Singer connection. If the constant is nonzero, then $g$ is K\"ahler and locally a complex space form.
\end{corollary}

The rest of the article is organized as follows. In Section~\ref{sec:prelim} we will set up the notations and collect some preliminaries results. Section~\ref{sec:parallel-curvature} will be devoted to the proof of Theorem~\ref{thm:partial-AS}. In Section~\ref{sec:block} we will discuss the holonomy decomposition and the corresponding block Ricci relations. Sections~\ref{sec:repeated-components} and~\ref{sec:three-block} deal with
block components having repeated block indices and three pairwise
distinct block indices, respectively. The final section completes
the proof of Theorem~\ref{thm:main} and gives the applications to
the Chern and Bismut connections.

\section{Preliminaries}\label{sec:prelim}

Throughout the paper, $(M^n,g)$ denotes a Hermitian manifold of complex dimension $n$, with complex structure $J$. Let $\nabla$ be a Hermitian connection. We extend $g$ complex bilinearly and write $\langle\,\cdot,\cdot\,\rangle$ for the resulting pairing. The torsion and curvature of $\nabla$ will be denoted by
$T(X,Y)=\nabla_XY-\nabla_YX-[X,Y]$ and $R(X,Y)Z=\nabla_X\nabla_YZ-\nabla_Y\nabla_XZ-\nabla_{[X,Y]}Z$.

Let $\{e_1,\ldots,e_n\}$ be a local unitary frame of $T^{1,0}M$. Unless stated otherwise, Latin indices range from $1$ to $n$. We use
$R_{i\bar j k\bar\ell}=\langle R(e_i,\bar e_j)e_k,\bar e_\ell\rangle$. Since $\nabla$ is Hermitian, curvature preserves type and satisfies $R_{i\bar j k\bar\ell}=\overline{R_{j\bar i\ell\bar k}}$.

For $0\neq X\in T^{1,0}M$, the holomorphic sectional curvature of
$\nabla$ in the direction of $X$ is $H(X)=R_{X\bar X X\bar X}/|X|^4$. Thus,
\begin{equation}\label{eq:HSC-polarization}
H\equiv c\quad\Longleftrightarrow\quad
\widehat R_{i\bar j k\bar\ell}=\frac c2(\delta_{ij}\delta_{k\ell}+\delta_{i\ell}\delta_{kj})
\end{equation}
with respect to any local unitary frame, where
\begin{equation*}
\widehat R_{i\bar j k\bar\ell}
=\frac14\big(
R_{i\bar j k\bar\ell}
+R_{k\bar j i\bar\ell}
+R_{i\bar\ell k\bar j}
+R_{k\bar\ell i\bar j}
\big)
\end{equation*}
is the symmetrization of $R$. If $R$ satisfies the K\"ahler curvature symmetries, then $\widehat R=R$. For a general Hermitian connection, however, this need not be the case, and the holomorphic sectional curvature determines only $\widehat R$, rather than the full curvature tensor $R$.

We also use the first Bianchi identity for a connection with torsion:
\begin{equation}\label{eq:first-Bianchi-general}
\mathfrak S_{X,Y,Z}\{R(X,Y)Z-(\nabla_ZT)(X,Y)-T(T(X,Y),Z)\}=0,
\end{equation}
where $\mathfrak S$ denotes cyclic summation; see \cite{Agricola,KN,ChenZBTP}. Since $\nabla J=0$ and the complex structure is integrable, the torsion satisfies
\begin{equation}\label{eq:torsion-type}
T(T^{1,0}M,T^{1,0}M)\subset T^{1,0}M.
\end{equation}

We shall also use the following form of Schur's lemma for finite-dimensional complex representations; see \cite{Humphreys}.

\begin{lemma}[Schur]\label{lem:Schur}
Let $V$ and $W$ be finite-dimensional irreducible complex representations of a group $G$. Any nonzero $G$-equivariant map $V\to W$ is an isomorphism. In particular, an endomorphism of $V$ commuting with $G$ is a complex scalar multiple of the identity map. If the representation is unitary and the endomorphism is Hermitian, the scalar is real. Consequently, every $G$-invariant Hermitian form on $V$ is a real multiple of the Hermitian metric.
\end{lemma}

This is the only form of Schur's lemma used below. It will be applied to invariant Hermitian forms and parallel maps between irreducible holonomy blocks.

\section{Parallelism of the \texorpdfstring{$(1,1)$}{(1,1)}-part of the curvature}\label{sec:parallel-curvature}

Throughout this section, $\nabla$ is the Hermitian connection fixed in Section~\ref{sec:prelim}, $\nabla T=0$, and $H$ is globally constant, say $H\equiv c$. We prove Theorem~\ref{thm:partial-AS}. For $X,Z\in T^{1,0}M$ and $\bar Y\in T^{0,1}M$, define
\begin{equation}\label{eq:Q-def}
\begin{split}
\mathcal Q(X,\bar Y,Z):=\big\{&T(T(X,\bar Y),Z)+T(T(\bar Y,Z),X)+T(T(Z,X),\bar Y)\big\}^{1,0}.
\end{split}
\end{equation}
The projection to type $(1,0)$ in \eqref{eq:Q-def} is $\nabla$-parallel because $\nabla J=0$.

\begin{lemma}\label{lem:curvature-difference}
If $\nabla T=0$, then
\begin{equation}\label{eq:curvature-difference}
R(X,\bar Y)Z-R(Z,\bar Y)X=\mathcal Q(X,\bar Y,Z).
\end{equation}
Equivalently, if $Q_{i\bar j k\bar\ell}=\langle\mathcal Q(e_i,\bar e_j,e_k),\bar e_\ell\rangle$, then
$$
R_{i\bar j k\bar\ell}-R_{k\bar j i\bar\ell}=Q_{i\bar j k\bar\ell}.
$$
\end{lemma}

\begin{proof}
Insert $(X,\bar Y,Z)$ into \eqref{eq:first-Bianchi-general}. Since $\nabla T=0$, all covariant-derivative terms disappear and we obtain
$$
\begin{aligned}
R(X,\bar Y)Z+R(\bar Y,Z)X+R(Z,X)\bar Y
={}&T(T(X,\bar Y),Z)+T(T(\bar Y,Z),X)\\
&+T(T(Z,X),\bar Y).
\end{aligned}
$$
Using the skew-symmetry of the curvature in its first two variables, $R(\bar Y,Z)X=-R(Z,\bar Y)X$. Because $\nabla$ is Hermitian, every curvature endomorphism preserves type. Thus the first two terms on the left lie in $T^{1,0}M$, whereas $R(Z,X)\bar Y\in T^{0,1}M$. Taking the $(1,0)$-part therefore gives
$$
R(X,\bar Y)Z-R(Z,\bar Y)X
=\big\{T(T(X,\bar Y),Z)+T(T(\bar Y,Z),X)+T(T(Z,X),\bar Y)\big\}^{1,0},
$$
which is \eqref{eq:curvature-difference}. Pairing with $\bar e_\ell$ gives the component form.
\end{proof}

\begin{proposition}\label{prop:curvature-reconstruction}
Assume that $\nabla T=0$ and $H\equiv c$. Then
\begin{equation}\label{eq:curvature-reconstruction}
\begin{split}
R_{i\bar j k\bar\ell}={}&
\frac c2\big(\delta_{ij}\delta_{k\ell}
+\delta_{i\ell}\delta_{kj}\big)
+\frac14Q_{i\bar j k\bar\ell}
+\frac12\overline{Q_{j\bar i\ell\bar k}}
+\frac14Q_{i\bar\ell k\bar j}
\end{split}
\end{equation}
for all $1\le i,j,k,\ell\le n$, with respect to any local unitary
frame.
\end{proposition}
\begin{proof}
By Lemma~\ref{lem:curvature-difference},
$$
R_{i\bar j k\bar\ell}-R_{k\bar j i\bar\ell}
=Q_{i\bar j k\bar\ell},
\qquad
R_{i\bar\ell k\bar j}-R_{k\bar\ell i\bar j}
=Q_{i\bar\ell k\bar j}.
$$
Applying the same lemma to $(e_j,\bar e_i,e_\ell)$, taking complex
conjugates, and using
$\overline{R_{j\bar i\ell\bar k}}=R_{i\bar j k\bar\ell}$, we also get
$$
R_{i\bar j k\bar\ell}-R_{i\bar\ell k\bar j}
=\overline{Q_{j\bar i\ell\bar k}}.
$$
Notice that no K\"ahler-type symmetry is assumed, so the conjugated term cannot in general be replaced by $Q_{i\bar j k\bar\ell}$.

Using these three identities in the symmetrization of $R$, we obtain
$$
4\widehat R_{i\bar j k\bar\ell}
=
4R_{i\bar j k\bar\ell}
-Q_{i\bar j k\bar\ell}
-2\overline{Q_{j\bar i\ell\bar k}}
-Q_{i\bar\ell k\bar j}.
$$
Now \eqref{eq:HSC-polarization} gives
$$
4R_{i\bar j k\bar\ell}
-Q_{i\bar j k\bar\ell}
-2\overline{Q_{j\bar i\ell\bar k}}
-Q_{i\bar\ell k\bar j}
=
2c(\delta_{ij}\delta_{k\ell}+\delta_{i\ell}\delta_{kj}),
$$
which is equivalent to \eqref{eq:curvature-reconstruction}.
\end{proof}

\begin{proof}[Proof of Theorem~\ref{thm:partial-AS}]
By Proposition~\ref{prop:curvature-reconstruction}, with respect to a local unitary frame we have
\begin{equation}\label{eq:parallel-curvature-proof}
\begin{split}
R_{i\bar j k\bar\ell}
={}&\frac c2
\bigl(\delta_{ij}\delta_{k\ell}
+\delta_{i\ell}\delta_{kj}\big)
+\frac14Q_{i\bar j k\bar\ell}+\frac12\overline{Q_{j\bar i\ell\bar k}}
+\frac14Q_{i\bar\ell k\bar j}.
\end{split}
\end{equation}

By \eqref{eq:Q-def}, $\mathcal Q$ is obtained algebraically from $T$ by tensor products, contractions, and the projection $\pi^{1,0}=\frac12(I-\sqrt{-1}J)$. The assumptions $\nabla T=0$ and $\nabla J=0$ therefore imply $\nabla\mathcal Q=0$. Since $\nabla g=0$, taking the metric component of $\mathcal Q$ gives $\nabla Q=0$, and complex conjugation gives $\nabla\overline Q=0$.

The constant $c$ is global, so $dc=0$. Since $\nabla g=0$, the first term on
the right-hand side of \eqref{eq:parallel-curvature-proof} is
$\nabla$-parallel as well. Consequently every term on the right-hand side
of \eqref{eq:parallel-curvature-proof} is $\nabla$-parallel. Hence
$\nabla(R^{1,1})=0$.

If $R$ is of type $(1,1)$, then $R=R^{1,1}$ and hence $\nabla R=0$. Together with $\nabla T=0$, this shows that both the
torsion and curvature of $\nabla$ are parallel. Thus $\nabla$ is an Ambrose--Singer connection.
\end{proof}

\section{Holonomy decomposition and normalized block Ricci forms}\label{sec:block}

Throughout this section, $\nabla T=0$ and $H$ is globally constant, say $H\equiv c$. Hence Theorem~\ref{thm:partial-AS} is available and the $(1,1)$-part of $R$ is $\nabla$-parallel.

Fix $p\in M$. Let $\Hol_p(\nabla)$ denote the holonomy group of $\nabla$ at $p$, and let $\Hol_p^0(\nabla)$ denote its identity component, equivalently, the restricted holonomy group generated by parallel transports along contractible loops based at $p$. Let $G=\overline{\Hol_p^0(\nabla)}\subset U(T_p^{1,0}M)$. Then $G$ is a compact connected Lie group. Since the holonomy representation is continuous, $G$ and $\Hol_p^0(\nabla)$ have the same invariant subspaces. Thus passing to the closure does not affect the irreducible decomposition used
below.

Choose an orthogonal decomposition into irreducible complex $G$-modules,
\begin{equation}\label{eq:holonomy-decomposition-point}
T_p^{1,0}M=V_1\oplus\cdots\oplus V_N.
\end{equation}
Equivalent irreducible summands are allowed. On a simply connected neighborhood of $p$, parallel transport extends every $V_\alpha$ to a $\nabla$-parallel subbundle, again denoted by $V_\alpha$. Its orthogonal projection $\pi_\alpha:T^{1,0}M\to V_\alpha$ is $\nabla$-parallel.

Viewing $T$ as a complex bilinear map, we define a \emph{block component} of the complexified torsion to be a projection of one of the forms
$$
\pi_k\circ T|_{V_i\otimes V_j}:V_i\otimes V_j\to V_k,\qquad
\pi_k\circ T|_{V_i\otimes\bar V_j}:V_i\otimes\bar V_j\to V_k,\qquad
\bar\pi_k\circ T|_{V_i\otimes\bar V_j}:V_i\otimes\bar V_j\to\bar V_k,
$$
together with their complex conjugates, where $\bar\pi_k$ denotes the
orthogonal projection onto $\bar V_k$. Since $T$ and all block
projections are $\nabla$-parallel, every block component is
$\nabla$-parallel and hence $G$-equivariant.

Put $n_\alpha=\operatorname{rank}_{\C}V_\alpha$. Define the \emph{normalized block Ricci form} by
\begin{equation*}
\rho_\alpha(X,\bar Y)=\frac1{n_\alpha}\tr_{V_\alpha}R(X,\bar Y).
\end{equation*}
Since $R^{1,1}$ and the block projections are $\nabla$-parallel, Theorem~\ref{thm:partial-AS} implies that every $\rho_\alpha$ is $\nabla$-parallel.

\begin{lemma}\label{lem:block-scalar}
For every pair $\alpha,\beta$, the restriction of $\rho_\beta$ to $V_\alpha$ is a real multiple of the Hermitian metric. More precisely, there exists a real constant $r_{\alpha\beta}$ such that
$$\rho_\beta|_{V_\alpha}=r_{\alpha\beta}\,g|_{V_\alpha}.$$
Equivalently, if $x\in V_\alpha$ is any unit vector and $\{u_1,\ldots,u_{n_\beta}\}$ is a unitary basis of $V_\beta$, then
$$
r_{\alpha\beta}:=\rho_\beta(x,\bar x)=\frac1{n_\beta}\sum_{a=1}^{n_\beta}R_{x\bar x u_a\bar u_a}.
$$
This number is independent of the choices of $x$ and of the unitary basis of $V_\beta$. Moreover, if $\alpha\neq\beta$, then
\begin{equation}\label{eq:cross-block-Ricci}
r_{\alpha\beta}+r_{\beta\alpha}=2c.
\end{equation}
\end{lemma}

\begin{proof}
By Theorem~\ref{thm:partial-AS}, $R^{1,1}$ is $\nabla$-parallel.
Since $V_\beta$ is $\nabla$-parallel, so is $\rho_\beta$. Hence
$\rho_\beta|_{V_\alpha}$ is a holonomy-invariant Hermitian form on
the irreducible block $V_\alpha$. By Lemma~\ref{lem:Schur},
$$\rho_\beta|_{V_\alpha} =r_{\alpha\beta}\,g|_{V_\alpha}$$
for some real constant $r_{\alpha\beta}$. The trace formula in the
statement follows directly from the definition of $\rho_\beta$ and
is independent of the choice of a unitary basis of $V_\beta$.

Now assume $\alpha\neq\beta$. Choose unitary bases
$\{e_1,\ldots,e_{n_\alpha}\}$ of $V_\alpha$ and
$\{u_1,\ldots,u_{n_\beta}\}$ of $V_\beta$. Since the curvature
preserves each parallel block and $V_\alpha\perp V_\beta$,
$$R_{e_i\bar u_a u_a\bar e_i} =R_{u_a\bar e_i e_i\bar u_a}=0.$$
Applying \eqref{eq:HSC-polarization} to $e_i$ and $u_a$ gives
$$R_{e_i\bar e_i u_a\bar u_a} +R_{u_a\bar u_a e_i\bar e_i}=2c.$$
Summing over $i$ and $a$ and using the definition of the normalized block Ricci forms yields
$$r_{\alpha\beta}+r_{\beta\alpha}=2c,$$
which is \eqref{eq:cross-block-Ricci}.
\end{proof}

We next state the block relations induced by a parallel block component.

\begin{lemma}\label{lem:block-Ricci-relations}
Let $E,F,W$ be irreducible parallel holonomy subbundles.
\begin{enumerate}
\item If a nonzero block component gives a map $\Phi:E\otimes F\to W$, then $\rho_W=\rho_E+\rho_F$.
\item If a nonzero block component gives a map $\Phi:E\otimes\bar F\to W$, then $\rho_W=\rho_E-\rho_F$.
\item If a nonzero block component gives a map $\Phi:E\otimes\bar F\to\bar W$, then $\rho_W=\rho_F-\rho_E$.
\item If a nonzero block component gives $\Phi:\Lambda^2E\to W$, then $\rho_W=2\rho_E$.
\end{enumerate}
\end{lemma}

\begin{proof}
We prove the first assertion in detail. Put $m_E=\operatorname{rank}_{\C}E$, $m_F=\operatorname{rank}_{\C}F$, and $m_W=\operatorname{rank}_{\C}W$. Since $\Phi$ is parallel, it is equivariant with respect to the holonomy action. Hence $\im\Phi\subset W$ is a holonomy-invariant subspace. Since $W$ is irreducible and $\Phi\neq0$, we have $\im\Phi=W$.

The endomorphism $\Phi\Phi^*:W\to W$ commutes with holonomy. By Lemma~\ref{lem:Schur},
$\Phi\Phi^*=\lambda I_W$
for some $\lambda\in\R$. Since $\Phi$ is surjective, $\Phi\Phi^*$ is positive definite, and hence $\lambda>0$.

The two partial traces $\tr_F(\Phi^*\Phi)\in\End(E)$ and $\tr_E(\Phi^*\Phi)\in\End(F)$ also commute with holonomy. Applying Lemma~\ref{lem:Schur} gives
$$\tr_F(\Phi^*\Phi)=\mu I_E,\qquad \tr_E(\Phi^*\Phi)=\nu I_F$$
for some real numbers $\mu$ and $\nu$. Taking traces and using
$\tr(\Phi^*\Phi)=\tr(\Phi\Phi^*)=\lambda m_W$, we obtain
$m_E\mu=\lambda m_W$ and $m_F\nu=\lambda m_W$. Therefore
$$
\tr_F(\Phi^*\Phi)=\frac{\lambda m_W}{m_E}I_E,\qquad
\tr_E(\Phi^*\Phi)=\frac{\lambda m_W}{m_F}I_F.
$$

For any $X,Y\in T^{1,0}M$, the equality $\nabla\Phi=0$ and the Ricci commutation formula give
$$
R(X,\bar Y)|_W\,\Phi
=\Phi\big(R(X,\bar Y)|_E\otimes I+I\otimes R(X,\bar Y)|_F\big).
$$
Right-multiplying by $\Phi^*$ and taking traces, the cyclicity of trace and the two partial-trace identities above yield
$$
\lambda\tr_W R(X,\bar Y)
=\frac{\lambda m_W}{m_E}\tr_E R(X,\bar Y)
+\frac{\lambda m_W}{m_F}\tr_F R(X,\bar Y).
$$
Dividing by $\lambda m_W$ gives
$$
\frac1{m_W}\tr_W R(X,\bar Y)
=\frac1{m_E}\tr_E R(X,\bar Y)
+\frac1{m_F}\tr_F R(X,\bar Y),
$$
which is precisely $\rho_W=\rho_E+\rho_F$.

For the second assertion, identify the conjugate unitary representation $\bar F$ with the dual representation $F^*$. The induced curvature on $F^*$ is the negative transpose of the curvature on $F$, and hence $\rho_{\bar F}=-\rho_F$. Applying the first assertion to $\Phi:E\otimes\bar F\to W$ gives $\rho_W=\rho_E-\rho_F$.

If instead $\Phi:E\otimes\bar F\to\bar W$, then $\rho_{\bar W}=-\rho_W$, and therefore $-\rho_W=\rho_E-\rho_F$, or equivalently, $\rho_W=\rho_F-\rho_E$. Finally, for $\Phi:\Lambda^2E\to W$, regard $\Phi$ as a skew-symmetric map on
$E\otimes E$. Since the curvature acts on the two copies of $E$
separately, the same trace computation gives two identical contributions. Hence $\rho_W=2\rho_E$.
\end{proof}

Let $\varepsilon_1,\ldots,\varepsilon_N$ denote the standard basis of $\R^N$. We call $a=(a_1,\ldots,a_N)\in\R^N$ a \emph{relation vector} if $\sum_\alpha a_\alpha\rho_\alpha=0$. Thus, for example, the relation $\rho_k=\rho_i+\rho_j$ has relation vector $\varepsilon_i+\varepsilon_j-\varepsilon_k$, while $\rho_j=2\rho_i$ has relation vector $2\varepsilon_i-\varepsilon_j$.

We now derive the main quadratic identity. Put $d_\alpha=r_{\alpha\alpha}$.

\begin{lemma}\label{lem:quadratic}
Let $a=(a_1,\ldots,a_N)\in\R^N$ be a relation vector. Then
\begin{equation}\label{eq:quadratic}
c\left(\sum_\alpha a_\alpha\right)^2
+\sum_\alpha(d_\alpha-c)a_\alpha^2=0.
\end{equation}
Moreover, if $b=(b_1,\ldots,b_N)\in\R^N$ is another relation
vector, then
\begin{equation}\label{eq:quadratic-polarized}
c\left(\sum_\alpha a_\alpha\right)
\left(\sum_\alpha b_\alpha\right)
+\sum_\alpha(d_\alpha-c)a_\alpha b_\alpha=0.
\end{equation}
\end{lemma}

\begin{proof}
Since $a$ is a relation vector, $\sum_\alpha a_\alpha\rho_\alpha=0$. Restricting to $V_\gamma$ and using Lemma~\ref{lem:block-scalar} gives $\sum_\alpha a_\alpha r_{\gamma\alpha}=0$ for every $\gamma$. Multiplying by $a_\gamma$, summing in $\gamma$, and separating diagonal and off-diagonal terms, we obtain
$$
0=\sum_{\gamma,\alpha}a_\gamma a_\alpha r_{\gamma\alpha}
 =\sum_\alpha r_{\alpha\alpha}a_\alpha^2
 +\sum_{\alpha<\beta}a_\alpha a_\beta
   (r_{\alpha\beta}+r_{\beta\alpha}).
$$
By the definition $d_\alpha=r_{\alpha\alpha}$ and
\eqref{eq:cross-block-Ricci},
$r_{\alpha\beta}+r_{\beta\alpha}=2c$ for $\alpha\neq\beta$.
Therefore
$$
0=
\sum_\alpha d_\alpha a_\alpha^2
+2c\sum_{\alpha<\beta}a_\alpha a_\beta.
$$
Using
$$
2\sum_{\alpha<\beta}a_\alpha a_\beta
=
\left(\sum_\alpha a_\alpha\right)^2
-\sum_\alpha a_\alpha^2,
$$
we obtain
$$
0=
c\left(\sum_\alpha a_\alpha\right)^2
+\sum_\alpha(d_\alpha-c)a_\alpha^2,
$$
which proves \eqref{eq:quadratic}.

Now let $b=(b_1,\ldots,b_N)$ be another relation vector. Since
the space of relation vectors is linear, $a+b$ is also a relation
vector. Applying \eqref{eq:quadratic} to $a+b$, and subtracting the
corresponding identities for $a$ and $b$, we obtain
$$
\begin{aligned}
0={}&
2c\left(\sum_\alpha a_\alpha\right)
\left(\sum_\alpha b_\alpha\right)
+2\sum_\alpha(d_\alpha-c)a_\alpha b_\alpha.
\end{aligned}
$$
Dividing by $2$ gives \eqref{eq:quadratic-polarized}.
\end{proof}

A block $V_\alpha$ is called \emph{clean} if $T(V_\alpha,V_\alpha)=0$ and $T(V_\alpha,\bar V_\alpha)=0$.

\begin{lemma}\label{lem:clean-block}
If $V_\alpha$ is clean, then
\begin{equation}\label{eq:clean-diagonal}
d_\alpha=\frac{n_\alpha+1}{2n_\alpha}c.
\end{equation}
\end{lemma}

\begin{proof}
Take $X,Z\in V_\alpha$ and $\bar Y\in\bar V_\alpha$ in the first Bianchi identity \eqref{eq:first-Bianchi-general}. Since $\nabla T=0$, the derivative terms vanish. The clean condition gives
$$T(X,Z)=0,\qquad T(X,\bar Y)=0,\qquad T(Z,\bar Y)=0,$$
and therefore every quadratic torsion term in the Bianchi identity vanishes. Taking the $(1,0)$-part, exactly as in the proof of Lemma~\ref{lem:curvature-difference}, yields $R(X,\bar Y)Z=R(Z,\bar Y)X$. Choose a unitary basis $\{e_1,\ldots,e_{n_\alpha}\}$ of $V_\alpha$. For $1\le i,j,k,\ell\le n_\alpha$, the above identity becomes $R_{i\bar j k\bar\ell}=R_{k\bar j i\bar\ell}$. Then $R_{i\bar j k\bar\ell}=R_{i\bar\ell k\bar j}$. The identity \eqref{eq:HSC-polarization} therefore gives, on this block,
$$
R_{i\bar j k\bar\ell}=\frac c2(\delta_{ij}\delta_{k\ell}+\delta_{i\ell}\delta_{kj}).
$$
Now contract $k$ and $\ell$ over a unitary basis of $V_\alpha$:
$$
\sum_{k=1}^{n_\alpha}R_{i\bar j k\bar k}
=\frac c2\sum_{k=1}^{n_\alpha}(\delta_{ij}+\delta_{ik}\delta_{kj})
=\frac{n_\alpha+1}{2}c\,\delta_{ij}.
$$
Dividing by $n_\alpha$ and using the definition of $\rho_\alpha$ gives
$$
\rho_\alpha|_{V_\alpha}=\frac{n_\alpha+1}{2n_\alpha}c\,g|_{V_\alpha}.
$$
Since $\rho_\alpha|_{V_\alpha}=d_\alpha g|_{V_\alpha}$, this gives \eqref{eq:clean-diagonal}.
\end{proof}

The following consequence of the Bianchi identity will be used to exclude pairs of vectors for which the relevant torsion terms vanish.

\begin{lemma}\label{lem:null-pair}
Let $V_\alpha$ and $V_\beta$ be distinct parallel blocks. Suppose that nonzero vectors $x\in V_\alpha$ and $y\in V_\beta$ satisfy
$$T(x,y)=T(x,\bar y)=T(y,\bar x)=T(x,\bar x)=T(y,\bar y)=0.$$
Then $c=0$.
\end{lemma}

\begin{proof}
Apply the first Bianchi identity \eqref{eq:first-Bianchi-general} to $(x,y,\bar y)$. Since $\nabla T=0$ and $T(x,y)=T(y,\bar y)=T(x,\bar y)=0$, all torsion terms vanish, and hence
$$R(x,y)\bar y+R(y,\bar y)x+R(\bar y,x)y=0.$$
Since curvature preserves every parallel block and its conjugate, the three terms belong to $\bar V_\beta$, $V_\alpha$, and $V_\beta$, respectively. Projecting onto $V_\alpha$ gives
$R(y,\bar y)x=0$, and therefore $R_{y\bar yx\bar x}=0$.
Interchanging $x$ and $y$ gives similarly $R_{x\bar xy\bar y}=0$.

Set $e_1=x/|x|$ and $e_2=y/|y|$. Since $V_\alpha\perp V_\beta$, these vectors are orthonormal and may be extended to a local unitary frame. Applying \eqref{eq:HSC-polarization} with the indices $1,1,2,2$ gives
$$
R_{e_1\bar e_1e_2\bar e_2}
+R_{e_2\bar e_1e_1\bar e_2}
+R_{e_1\bar e_2e_2\bar e_1}
+R_{e_2\bar e_2e_1\bar e_1}
=2c.
$$
The middle two terms vanish because curvature preserves the two
orthogonal blocks, while the first and last vanish by the previous Bianchi computations. Hence $2c=0$, and therefore $c=0$.
\end{proof}

\section{Torsion components with repeated block indices}\label{sec:repeated-components}

Throughout this section, we assume that $\nabla T=0$ and $H\equiv c$, with $c\neq0$. We use the holonomy decomposition and the normalized block Ricci forms introduced in Section~\ref{sec:block}. We first describe all possible block components of the torsion with a repeated underlying block index.

Recall that $G=\overline{\Hol_p^0(\nabla)}$ is compact and connected,
and that
$T_p^{1,0}M=V_1\oplus\cdots\oplus V_N$
is an orthogonal decomposition into irreducible $G$-modules. Let
$Z(G)$ denote the center of $G$. For each $z\in Z(G)$, Lemma~\ref{lem:Schur} gives
$z|_{V_\alpha}=\chi_\alpha(z)I_{V_\alpha}$
for some $\chi_\alpha(z)\in U(1)$. Thus
$\chi_\alpha:Z(G)\to U(1)$ is a unitary character, which we call the
\emph{central character} of $V_\alpha$. We say that it is
\emph{trivial} if $\chi_\alpha\equiv1$ on $Z(G)$. The conjugate
module $\bar V_\alpha$ has central character
$\overline{\chi_\alpha}=\chi_\alpha^{-1}$.

Every block component is $G$-equivariant. Therefore the
central characters satisfy the corresponding multiplicative
relations. For block indices $1\le i,j,k\le N$,
\begin{equation}\label{eq:central-character-relations}
\begin{aligned}
V_i\otimes V_j\to V_k
&\quad\Longrightarrow\quad
\chi_k=\chi_i\chi_j,\\
V_i\otimes\bar V_j\to V_k
&\quad\Longrightarrow\quad
\chi_k=\chi_i\chi_j^{-1},\\
V_i\otimes\bar V_j\to\bar V_k
&\quad\Longrightarrow\quad
\chi_k^{-1}=\chi_i\chi_j^{-1}.
\end{aligned}
\end{equation}
Indeed, if $\Phi:V_i\otimes V_j\to V_k$ is nonzero and
$z\in Z(G)$, then equivariance gives
$\chi_k(z)\Phi(x,y) = \Phi(\chi_i(z)x,\chi_j(z)y)$.
Choosing $x,y$ with $\Phi(x,y)\neq0$ implies
$\chi_k=\chi_i\chi_j$. The other two identities follow in the same
way.

We now consider all possible repetitions among the block indices.
For a pure component $V_i\otimes V_j\to V_k$, if $k=i$ or $k=j$,
then \eqref{eq:central-character-relations} shows that the central
character of the other input block is trivial. After excluding these
cases, the only possible repetition is $i=j$ with $k\neq i$, and
Lemma~\ref{lem:block-Ricci-relations} gives
$\rho_k=2\rho_i$. Notice that when $i=j=k$, the relation
$\chi_i=\chi_i^2$ also gives $\chi_i\equiv1$.

For a mixed component $V_i\otimes\bar V_j\to V_k$, the cases
$i=j$ and $k=i$ give a block with trivial central character. The only
remaining repeated-index case is $k=j$ with $i\neq j$, for which
Lemma~\ref{lem:block-Ricci-relations} gives
$\rho_i=2\rho_j$. Similarly, for
$V_i\otimes\bar V_j\to\bar V_k$, the cases $i=j$ and $k=j$
give a trivial central character, while the remaining possibility
$k=i$ with $i\neq j$ gives
$\rho_j=2\rho_i$.

Thus every block component with a repeated underlying block index either gives a block with trivial central character or gives a \emph{doubling relation} of the form $\rho_j=2\rho_i$ for some $i\neq j$.

\begin{lemma}\label{lem:trivial-central-character}
Let $V$ be an irreducible parallel holonomy block whose central
character is trivial on $Z(G)$. Then $c=0$.
\end{lemma}

\begin{proof}
We first consider the central and semisimple parts of the holonomy
representation separately. At this point, we use the fact that the central character is trivial on the whole center $Z(G)$.

Let $\mathfrak g$ denote the Lie algebra of $G$. Since $G$ is compact and connected, its Lie algebra admits the decomposition
$\mathfrak g=\mathfrak z\oplus\mathfrak g_s$,
where $\mathfrak z$ is the center of $\mathfrak g$ and
$\mathfrak g_s=[\mathfrak g,\mathfrak g]$ is compact semisimple.
Let $G_s=[G,G]$. Then $G=Z(G)^0G_s$, where $Z(G)^0$ is the identity component of the center $Z(G)$. For every $z\in Z(G)^0$, the action of $z$ on the irreducible $G$-module $V$ is given by a scalar. It follows that any complex subspace of $V$ invariant under $G_s$ is automatically invariant under $Z(G)^0$ as well. Since $G$ is generated by $Z(G)^0$ and $G_s$, such a subspace is in fact $G$-invariant. The irreducibility of $V$ as a $G$-module therefore shows that $V$ remains irreducible when restricted to $G_s$.

If $G_s$ is trivial, then the irreducibility of $V$ as a
$G_s$-module implies that $\dim_\C V=1$. Moreover, $G=Z(G)^0$, and the triviality of the central character implies that $G$ acts trivially on $V$. Hence the holonomy algebra acts trivially
on $V$. For any unit vector $v\in V$, we therefore have $R(v,\bar v)v=0$, and thus $H(v)=0$. Since $H\equiv c$, it follows that $c=0$.

We may therefore assume that $G_s$ is nontrivial.

Let $\pi:\widetilde G_s\longrightarrow G_s$
be the simply connected compact semisimple covering group of $G_s$. Since the central character of $V$ is trivial on $Z(G)$, every element of $Z(G)$ acts trivially on $V$. In particular, the action of $Z(G)^0$ is trivial, and differentiation shows that the center $\mathfrak z$ of $\mathfrak g$ acts trivially as well.

We next consider the center of $\widetilde G_s$. Since $\pi$ is a
group homomorphism, it maps $Z(\widetilde G_s)$ into $Z(G_s)$.
On the other hand, every element of $Z(G_s)$ commutes with $G_s$ and
also with $Z(G)^0$, since $Z(G)^0$ is central in $G$. Since
$G=Z(G)^0G_s$, it follows that $Z(G_s)\subset Z(G)$. Therefore every element of $Z(\widetilde G_s)$ acts trivially on $V$
through the lifted representation. Hence the lifted irreducible
representation of $\widetilde G_s$ also has trivial central character.

Write $\widetilde G_s=G_1\times\cdots\times G_r$ as a product of
simply connected compact simple groups. The irreducible representation $V$ is the external tensor product $V_1\boxtimes\cdots\boxtimes V_r$ of irreducible representations of the factors; its underlying vector space is $V_1\otimes\cdots\otimes V_r$. Since the center of $\widetilde G_s$ acts trivially on $V$, the center of each $G_a$ acts trivially on $V_a$. Equivalently, the highest weight of each $V_a$ lies
in the root lattice of $G_a$. By \cite[Remark 1]{LeFlochSmilga}, each $V_a$ has a nontrivial zero-weight space. Hence $V$ also has a nontrivial zero-weight space. Choose a nonzero zero-weight vector $v\in V$ and normalize it so that $|v|=1$.

Choose a maximal torus of $G_s$ with Lie algebra $\mathfrak t$, and set $\mathfrak h=\mathfrak t^\C\subset\mathfrak g_s^\C$. Then
\begin{equation*}
\mathfrak g_s^\C
=\mathfrak h\oplus\bigoplus_{\gamma\in\Delta}\mathfrak g_\gamma
\end{equation*}
is the corresponding root-space decomposition. Since $v$ has weight zero, $Hv=0$ for every $H\in\mathfrak h$. If $A_\gamma\in\mathfrak g_\gamma$, then $A_\gamma v$ is either zero or lies in the weight-$\gamma$ space. Since the representation is
unitary, distinct weight spaces are orthogonal. Hence
$\langle A_\gamma v,\bar v\rangle=0$.
It follows from the root-space decomposition that
$\langle Av,\bar v\rangle=0$
for every $A\in\mathfrak g_s^\C$. Moreover, $\mathfrak z$ acts
trivially on $V$, and therefore the same holds for
$A\in\mathfrak z^\C$. Thus
$$
\langle Av,\bar v\rangle=0
\qquad\text{for every }A\in\mathfrak g^\C.
$$
By the Ambrose--Singer holonomy theorem
\cite[Chapter II, Theorem 8.1]{KN}, the curvature endomorphisms at
the base point belong to the holonomy algebra. After complexification,
we have
$$
R(v,\bar v)|_V\in\mathfrak g^\C.
$$
Applying the previous identity with $A=R(v,\bar v)|_V$, we obtain
$$
R_{v\bar vv\bar v}
=\langle R(v,\bar v)v,\bar v\rangle=0.
$$
Since $|v|=1$, this gives $H(v)=0$. Since $H\equiv c$, we conclude
that $c=0$.
\end{proof}

We call a nonzero block component a \emph{doubling component} if
the relation given by Lemma~\ref{lem:block-Ricci-relations} is of the form
$$\rho_j=2\rho_i,\qquad i\neq j.$$
Such a relation arises, for example, from a pure component $\Lambda^2V_i\to V_j$, or from a mixed component $V_j\otimes\bar V_i\to V_i$.

For the next lemma, we introduce a notion of connectivity among the
holonomy blocks. The \emph{support} of a block component is the set of
underlying block indices occurring in its domain and target. Consider
the finite hypergraph whose vertices are the blocks $V_\alpha$ and whose
hyperedges are the supports of the nonzero block components. Two blocks are said to be
\emph{torsion-connected} if they belong to the same connected
component of this hypergraph. We refer to these connected components
as \emph{torsion-connected components}.

Under the standing assumption $c\neq0$, Lemma~\ref{lem:trivial-central-character} shows that a block component
giving a block with trivial central character cannot occur. Thus every
remaining nonzero block component leads either to a doubling
relation or to a relation involving three pairwise distinct blocks.

\begin{lemma}\label{lem:doubling-rank-collapse}
Assume $c\neq0$. Suppose that there is a doubling relation $\rho_2=2\rho_1$ within a torsion-connected component. Then all normalized block Ricci forms associated with the blocks in
that component are proportional.
\end{lemma}

\begin{proof}
We divide the proof into four steps.

\smallskip
\noindent\emph{Step 1: the directed doubling graph contains no directed cycle.}
Associate to each doubling relation $\rho_j=2\rho_i$ a directed edge
$i\to j$. Suppose that there is a directed cycle
$i_1\to i_2\to\cdots\to i_r\to i_1$.
Iterating the doubling relations along the cycle gives
$$
\rho_{i_2}=2\rho_{i_1},\qquad
\rho_{i_3}=2^2\rho_{i_1},\qquad\ldots,\qquad
\rho_{i_1}=2^r\rho_{i_1}.
$$
Since $2^r\neq1$, it follows that $\rho_{i_1}=0$, and hence
$\rho_{i_a}=0$ for every $1\le a\le r$.

For an edge $i\to j$ on the cycle, the corresponding relation vector
is $2\varepsilon_i-\varepsilon_j$, whose coefficient sum is $1$.
Lemma~\ref{lem:quadratic} therefore gives
$$c+4(d_i-c)+(d_j-c)=0,$$
that is,
$$4d_i+d_j=4c.$$
Since $\rho_i=\rho_j=0$, we have $d_i=d_j=0$, and hence $c=0$,
contrary to our assumption. Therefore the directed doubling graph is acyclic. As the graph is finite, each nonempty component contains a sink. We may therefore choose a doubling edge and relabel its endpoints so that
$$1\longrightarrow2,\qquad \rho_2=2\rho_1,$$
with $2$ a sink, namely, no doubling edge starts at $2$.

\smallskip
\noindent\emph{Step 2: the sink block $V_2$ is clean.}
We show that
$T(V_2,V_2)=T(V_2,\bar V_2)=0$.

Suppose first that $T(V_2,V_2)\neq0$. By the type property
\eqref{eq:torsion-type}, there is a nonzero block component of
the form
$$\Lambda^2V_2\longrightarrow V_k.$$
If $k\neq2$, Lemma~\ref{lem:block-Ricci-relations} gives
$\rho_k=2\rho_2$, and hence a doubling edge $2\to k$, contradicting the choice of $2$ as a sink. If $k=2$, the corresponding character relation is $\chi_2=\chi_2^2$, which implies $\chi_2\equiv1$. Lemma~\ref{lem:trivial-central-character} then gives $c=0$, again a contradiction. Therefore $T(V_2,V_2)=0$.

Suppose next that $T(V_2,\bar V_2)\neq0$. Then there is a nonzero
block component taking values in some $V_k$ or $\bar V_k$.
Since $\chi_2\chi_2^{-1}=1$, the character relations \eqref{eq:central-character-relations} imply $\chi_k\equiv1$ in either case. Lemma~\ref{lem:trivial-central-character} would again give $c=0$. Therefore $T(V_2,\bar V_2)=0$, and hence $V_2$ is clean.

By Lemma~\ref{lem:clean-block},
$$
d_2=\alpha c,
\qquad
\alpha=\frac{n_2+1}{2n_2}
=\frac12+\frac1{2n_2}\in\left(\frac12,1\right].
$$
The doubling relation $2\rho_1-\rho_2=0$ has relation vector
$u=2\varepsilon_1-\varepsilon_2$, namely $u=(2,-1,0,\ldots,0)\in\R^N$, whose coefficient sum is $1$.
Applying Lemma~\ref{lem:quadratic} to $u$ gives $c+4(d_1-c)+(d_2-c)=0$, or equivalently $4d_1+d_2=4c$. Since $d_2=\alpha c$, we obtain $$d_1=\left(1-\frac{\alpha}{4}\right)c.$$

\smallskip
\noindent\emph{Step 3: the third Ricci form in every three-block
relation is proportional to $\rho_1$.}
Let $v=(v_1,\ldots,v_N)$ be any relation vector arising from a
nonzero block component in the same torsion-connected component, and
set $s(v)=\sum_{i=1}^N v_i$. Applying
\eqref{eq:quadratic-polarized} to
$u=2\varepsilon_1-\varepsilon_2$ and $v$, and using
$u_1=2$, $u_2=-1$, and $\sum_i u_i=1$, we obtain
$$
c\,s(v)+2(d_1-c)v_1-(d_2-c)v_2=0.
$$
Since
$d_1-c=-\frac{\alpha}{4}c$ and
$d_2-c=(\alpha-1)c$, and since $c\neq0$, it follows that
\begin{equation}\label{eq:rank-collapse-equation}
s(v)=\frac{\alpha}{2}v_1+(\alpha-1)v_2.
\end{equation}
We first apply \eqref{eq:rank-collapse-equation} to a relation vector arising from a three-block component. By Lemma~\ref{lem:block-Ricci-relations}, such a vector has exactly three
nonzero coefficients, each equal to $\pm1$, and $s(v)=\pm1$. We
claim that its support contains both $1$ and $2$. If its support is disjoint from $\{1,2\}$, then the right-hand side of \eqref{eq:rank-collapse-equation} is zero. If it contains $1$ but not $2$, the absolute value of the right-hand side is
$\alpha/2\le1/2$. If it contains $2$ but not $1$, the absolute value is $|\alpha-1|<1/2$. Each case contradicts $|s(v)|=1$.

Thus the support of $v$ contains both $1$ and $2$. Let $k$ be the
third block index. The corresponding relation is of the form $$\pm\rho_1\pm\rho_2\pm\rho_k=0.$$
Since $\rho_2=2\rho_1$, we obtain
$$\rho_k\in\R\rho_1.$$
Therefore the normalized block Ricci form corresponding to the third block is also proportional to $\rho_1$.

\smallskip
\noindent\emph{Step 4: every additional doubling relation gives a Ricci form proportional to $\rho_1$.}
Let
$$v=2\varepsilon_p-\varepsilon_q,\qquad p\neq q,$$
be another doubling relation in the same torsion-connected component. Since $s(v)=1$, equation \eqref{eq:rank-collapse-equation} becomes
$$1=\frac{\alpha}{2}v_1+(\alpha-1)v_2.$$
We consider the possible positions of $p$ and $q$ relative to
$\{1,2\}$.

If $\{p,q\}\cap\{1,2\}=\varnothing$, then $v_1=v_2=0$, and the
right-hand side is zero. If $q=1$ and $p\neq2$, then
$(v_1,v_2)=(-1,0)$, and the right-hand side is $-\alpha/2$.
If $p=2$ and $q\neq1$, then $(v_1,v_2)=(0,2)$, and the
right-hand side is $2(\alpha-1)\le0$. If $q=2$ and $p\neq1$,
then $(v_1,v_2)=(0,-1)$, and the right-hand side is
$1-\alpha<1/2$. Finally, for the reverse relation $2\to1$, we have
$(v_1,v_2)=(-1,2)$, and the right-hand side is
$3\alpha/2-2\le-1/2$.
None of these values is equal to $1$.

Apart from the original relation $\rho_2=2\rho_1$, the only remaining
possibility is $p=1$ and $q\notin\{1,2\}$. In this case
$(v_1,v_2)=(2,0)$, so the above equation gives $\alpha=1$. The
doubling relation is
$$\rho_q=2\rho_1,$$
and therefore $\rho_q\in\R\rho_1$.

Combining Steps 3 and 4, every block $V_\alpha$ in this
torsion-connected component satisfies
$$\rho_\alpha\in\R\rho_1.$$
Indeed, every block in the component belongs to the support of some nonzero block component, and such a component gives either a three-block relation or a doubling relation. Hence all normalized block Ricci forms in the component are proportional. This completes the proof.
\end{proof}

We shall use the following elementary consequence twice. It follows from the
projective dimension theorem applied to the Segre variety; see
\cite[Chapter I, Theorem 7.2]{Hartshorne}.

\begin{lemma}\label{lem:Segre-dimension}
Let $\Phi:U\otimes W\to E$ be a complex linear map between nonzero
complex vector spaces. Suppose that
$\Phi(u\otimes w)\neq0$
whenever $u\neq0$ and $w\neq0$. Then
\begin{equation}\label{eq:Segre-dimension-estimate}
\dim_\C E\ge \dim_\C U+\dim_\C W-1.
\end{equation}
\end{lemma}

\begin{proof}
If $\ker\Phi=0$, then
$\dim_\C E\ge\dim_\C(U\otimes W)$,
and the conclusion is immediate. Suppose that $\ker\Phi\neq0$.
The assumption on $\Phi$ implies that
$$
\mathbb P(\ker\Phi)\cap
\big(\mathbb P(U)\times\mathbb P(W)\big)=\varnothing,
$$
where $\mathbb P(U)\times\mathbb P(W)$ is regarded as the Segre
variety in $\mathbb P(U\otimes W)$. By the projective dimension
theorem \cite[Chapter I, Theorem 7.2]{Hartshorne},
$$
\dim\mathbb P(\ker\Phi)
+\dim\big(\mathbb P(U)\times\mathbb P(W)\big)
<
\dim\mathbb P(U\otimes W).
$$
Since
$$\dim\mathbb P(\ker\Phi) \ge \dim_\C U\,\dim_\C W-\dim_\C E-1$$
and
$$
\dim\big(\mathbb P(U)\times\mathbb P(W)\big)
=\dim_\C U+\dim_\C W-2,
$$
we obtain
$\dim_\C E>\dim_\C U+\dim_\C W-2$.
Since the dimensions are integers,
\eqref{eq:Segre-dimension-estimate} follows.
\end{proof}
We now show that the proportionality obtained in Lemma~\ref{lem:doubling-rank-collapse} is incompatible with $c\neq0$.

\begin{proposition}\label{prop:doubling-exclusion}
Under the standing assumption $c\neq0$, no torsion-connected component can contain a doubling relation.
\end{proposition}
\begin{proof}
Assume that $c\neq0$. After relabeling, write the blocks in the relevant
torsion-connected component as $V_1,\ldots,V_m$, with
$\rho_2=2\rho_1$.
We first claim that $\rho_1\neq0$. Indeed, if $\rho_1=0$, then
$\rho_2=0$, and hence $d_1=d_2=0$. On the other hand, applying Lemma~\ref{lem:quadratic} to the relation
vector $2\varepsilon_1-\varepsilon_2$ gives $4d_1+d_2=4c$,
which contradicts $c\neq0$.

By Lemma~\ref{lem:doubling-rank-collapse}, all normalized block Ricci
forms in this component are proportional to $\rho_1$. Since
$\rho_1\neq0$, for each $1\le i\le m$ there is a unique
$k_i\in\R$ such that
$$
\rho_i=k_i\rho_1.
$$
In particular, $k_1=1$ and $k_2=2$. Let $\rho=\rho_1$ and
$g_i=g|_{V_i}$. By Lemma~\ref{lem:block-scalar}, there is a real
number $s_i$ such that $\rho|_{V_i}=s_i g_i$. It follows that
$$
r_{ij}=k_js_i,\qquad r_{ji}=k_is_j,\qquad d_i=k_is_i.
$$
Therefore the cross-block identity
$r_{ij}+r_{ji}=2c$ becomes
\begin{equation}\label{eq:rank-one-pair}
k_js_i+k_is_j=2c,\qquad i\neq j.
\end{equation}

We next observe that $k_i\neq0$ for every $i$. Indeed, if $k_i=0$,
then necessarily $i\notin\{1,2\}$. Applying
\eqref{eq:rank-one-pair} to the pairs $(i,1)$ and $(i,2)$ gives
$$
s_i=2c,\qquad 2s_i=2c,
$$
respectively, which is impossible since $c\neq0$.

We also note that any block $V_i$ for which $|k_i|$ is maximal is
clean. Suppose first that there is a nonzero block component
$\Lambda^2V_i\to V_j$. Lemma~\ref{lem:block-Ricci-relations} gives
$\rho_j=2\rho_i$, and hence $k_j=2k_i$. If $j\neq i$, this
contradicts the maximality of $|k_i|$, while $j=i$ would give
$k_i=0$. Similarly, a nonzero block component of
$T(V_i,\bar V_i)$ would give a block Ricci form equal to zero,
contradicting $k_j\neq0$ for every block in the component.
Therefore $T(V_i,V_i)=T(V_i,\bar V_i)=0$, so $V_i$ is clean.

We now consider the possible number of blocks in the component.

\smallskip
\noindent\emph{Case 1: the component contains at least four blocks.}
Assume $m\ge4$. Choose three distinct indices $i,j,k$.
Equation \eqref{eq:rank-one-pair} gives
$$
k_js_i+k_is_j=2c,\qquad
k_ks_i+k_is_k=2c,\qquad
k_ks_j+k_js_k=2c.
$$
Since all $k_i$ are nonzero, the first two equations give
$$
s_j=\frac{2c-k_js_i}{k_i},
\qquad
s_k=\frac{2c-k_ks_i}{k_i}.
$$
Substituting these into the third equation gives
$$
2c(k_j+k_k)-2k_jk_ks_i=2ck_i,
$$
and hence
\begin{equation}\label{eq:si-triple}
s_i=c\,\frac{k_j+k_k-k_i}{k_jk_k}.
\end{equation}

Let $\ell$ be a fourth index. Applying
\eqref{eq:si-triple} to the triples $(i,j,k)$ and $(i,j,\ell)$
and comparing the two expressions for $s_i$, we obtain
\begin{equation}\label{eq:four-k-relation}
(k_\ell-k_k)(k_j-k_i)=0
\end{equation}
whenever $i,j,k,\ell$ are pairwise distinct.

Taking $i=1$ and $j=2$, and using $k_1=1$ and $k_2=2$, we see that
$k_3=\cdots=k_m$. Write $k_3=\cdots=k_m=r$.
Applying \eqref{eq:four-k-relation} with
$(i,j,k,\ell)=(3,1,2,4)$ gives
$$
(r-2)(1-r)=0.
$$
Thus $r=1$ or $r=2$.

Suppose first that $r=1$. Since $m\ge4$, there are at least three
blocks with coefficient $1$. For any two such blocks,
\eqref{eq:rank-one-pair} gives $s_i+s_j=2c$. Comparing the equations
for three coefficient-one blocks shows that their $s$-values are all
equal to $c$. Pairing any one of these blocks with $V_2$ then gives
$s_2=0$, and hence $d_2=k_2s_2=0$.

Since $|k_2|=2$ is maximal, $V_2$ is clean. Lemma~\ref{lem:clean-block} therefore gives
$$
\frac{d_2}{c}
=\frac{n_2+1}{2n_2}
\in\left(\frac12,1\right],
$$
contradicting $d_2=0$.

Suppose next that $r=2$. Then $k_2=\cdots=k_m=2$.
For distinct $i,j\ge2$, equation \eqref{eq:rank-one-pair} gives
$s_i+s_j=c$. Since there are at least three indices in
$\{2,\ldots,m\}$, it follows that $s_i=c/2$ for every $i\ge2$.
Consequently, $d_i=k_is_i=c$ for all $i\ge2$. These blocks have maximal $|k_i|$ and hence are
clean. Lemma~\ref{lem:clean-block} now gives
$$
1=\frac{d_i}{c}
=\frac{n_i+1}{2n_i},
$$
so $n_i=1$ for every $i\ge2$.

Choose two distinct blocks $V_i$ and $V_j$ with $i,j\ge2$, and let
$x\in V_i$ and $y\in V_j$ be unit vectors. Since both blocks are
clean, their self-torsion terms vanish. Moreover, a nonzero pure
component between $V_i$ and $V_j$ would, by Lemma~\ref{lem:block-Ricci-relations}, give an output block with normalized
Ricci form
$\rho_i+\rho_j=4\rho_1$.
This is impossible, since the proportionality coefficients $k_\alpha$ occurring in this torsion-connected component are only $1$ and $2$. Likewise, a nonzero mixed component
between these two blocks would give an output block Ricci form equal
to
$\rho_i-\rho_j=0$
or
$\rho_j-\rho_i=0$,
contradicting the nonvanishing of all $k_\alpha$. Hence all the
torsion terms required in Lemma~\ref{lem:null-pair} vanish. That
lemma gives $c=0$, again a contradiction.

\smallskip
\noindent\emph{Case 2: the component contains exactly three blocks.}
Write $(k_1,k_2,k_3)=(1,2,t)$. Since $V_3$ belongs to the same torsion-connected component as $V_1$ and $V_2$, there is a nonzero block component whose support contains $V_3$ and at least one of $V_1,V_2$. By Lemma~\ref{lem:trivial-central-character}, a relation giving a block with trivial central character cannot occur when $c\neq0$. Thus the remaining possibilities are a three-block relation or a doubling relation.

If the relation involves all three blocks, Lemma~\ref{lem:block-Ricci-relations} gives one of the sum or difference relations among $\rho_1,\rho_2,\rho_3$. Since
$\rho_2=2\rho_1$ and $\rho_3=t\rho_1$, this gives
$$t=3,\qquad t=1,\qquad\text{or}\qquad t=-1.$$
If instead the relation is a doubling relation involving $V_3$ and
one of $V_1,V_2$, then
$$t=2,\qquad t=\frac12,\qquad t=4,\qquad\text{or}\qquad t=1.$$
Consequently,
$$t\in\left\{-1,\frac12,1,2,3,4\right\}.$$
The three equations in \eqref{eq:rank-one-pair} are
$$2s_1+s_2=2c,\qquad ts_1+s_3=2c,\qquad ts_2+2s_3=2c.$$
Solving this system gives
$$
s_1=\frac{t+1}{2t}c,\qquad
s_2=\frac{t-1}{t}c,\qquad
s_3=\frac{3-t}{2}c.
$$
Since $d_i=k_is_i$, we obtain
\begin{equation}\label{eq:three-k-diagonal}
d_1=\frac{t+1}{2t}c,\qquad
d_2=\frac{2(t-1)}{t}c,\qquad
d_3=\frac{t(3-t)}{2}c.
\end{equation}

We now consider the possible values of $t$. If $t=-1$, then
$|k_2|$ is maximal and \eqref{eq:three-k-diagonal} gives
$d_2/c=4$. If $t=1/2$, then $|k_2|$ is again maximal and
$d_2/c=-2$. If $t=1$, then $|k_2|$ is maximal and $d_2/c=0$.
If $t=3$, then $|k_3|$ is maximal and $d_3/c=0$. Finally, if
$t=4$, then $|k_3|$ is maximal and $d_3/c=-2$. In each case the
corresponding block is clean, whereas Lemma~\ref{lem:clean-block}
gives
$$\frac{d_i}{c} =\frac{n_i+1}{2n_i} \in\left(\frac12,1\right].$$
Hence none of these five values of $t$ can occur.

It remains to consider $t=2$. In this case $|k_2|=|k_3|=2$ is
maximal and
$$d_2=d_3=c.$$
Thus both $V_2$ and $V_3$ are clean. By Lemma~\ref{lem:clean-block},
$1=\frac{d_i}{c} =\frac{n_i+1}{2n_i}, \qquad i=2,3$,
and hence $n_2=n_3=1$.

Choose unit vectors $x\in V_2$ and $y\in V_3$. Since both blocks
are clean, their self-torsion terms vanish. A nonzero pure block component involving $V_2$ and $V_3$ would give, by Lemma~\ref{lem:block-Ricci-relations}, an output block with normalized Ricci form
$$\rho_2+\rho_3=4\rho_1.$$
No such block occurs in this torsion-connected component. Likewise, a nonzero mixed component involving $V_2$ and $V_3$ would give an output block Ricci form equal to
$\rho_2-\rho_3=0$ or $\rho_3-\rho_2=0$, contradicting the fact that all $k_i$ are nonzero. Hence the torsion terms appearing in Lemma~\ref{lem:null-pair} vanish for $x$ and $y$. That lemma gives $c=0$, a contradiction. Thus Case 2 cannot occur.

\smallskip
\noindent\emph{Case 3: the component contains exactly two blocks.}
Write the two blocks as $A=V_1$ and $B=V_2$, so that $\rho_B=2\rho_A$. We consider two subcases depending on whether there is a nonzero pure component $\Lambda^2A\longrightarrow B$.

\smallskip
\noindent\emph{Case 3a: there is a nonzero pure component
$\Lambda^2A\to B$.}
Under the assumption $c\neq0$, the central-character argument at the beginning of this section shows that all same-block mixed components vanish. Moreover, a nonzero pure component $A\otimes B\to V$ would give, by Lemma~\ref{lem:block-Ricci-relations},
$$\rho_V=\rho_A+\rho_B=3\rho_A.$$
Since the torsion-connected component contains only the two blocks
$A$ and $B$, this is impossible. Thus, up to skew-symmetry and
complex conjugation, the only possible additional mixed component
compatible with $\rho_B=2\rho_A$ is
$$\Psi:B\otimes\bar A\longrightarrow A,$$
since in this case Lemma~\ref{lem:block-Ricci-relations} gives
$$\rho_A=\rho_B-\rho_A.$$

We claim that $\Psi$ is nonzero. Otherwise, for unit vectors
$x\in A$ and $\xi\in B$, all the torsion terms appearing in Lemma~\ref{lem:null-pair} would vanish: the same-block mixed terms vanish by the preceding discussion, the pure component $A\otimes B$ is zero, and the remaining mixed component is absent. Lemma~\ref{lem:null-pair} would then give $c=0$, a contradiction.

In fact, $\Psi$ has no nonzero decomposable element in its kernel.
Suppose that $\Psi(b,\bar a)=0$ for some nonzero $a\in A$ and $b\in B$. By complex conjugation and skew-symmetry, the corresponding component $T(a,\bar b)$ also vanishes. Together with the vanishing of the same-block mixed components and of the pure component $A\otimes B$, this makes $(a,b)$ a null pair. Lemma~\ref{lem:null-pair} again gives $c=0$, a contradiction.

Let $m=\dim_\C A$ and
$q=\dim_\C B$. Applying Lemma~\ref{lem:Segre-dimension} to $\Psi:B\otimes\bar A\to A$ gives $m\ge q+m-1$. Hence $q=1$. Choose a unit vector $\xi$ spanning $B$, so that $B=\C\xi$.

The nonzero pure component $\Lambda^2A\to B$ can now be written as
$T(x,y)=\mu\,\Omega(x,y)\xi, \qquad x,y\in A$, where $\mu\neq0$ and $\Omega\in\Lambda^2A^*$ is a nonzero alternating form. Its radical
$$
\operatorname{Rad}(\Omega)
=
\{x\in A:\Omega(x,y)=0\text{ for every }y\in A\}
$$
is holonomy invariant. Since $A$ is irreducible and $\Omega\neq0$, we have $\operatorname{Rad}(\Omega)=0$. Thus $\Omega$ is nondegenerate, and in particular $m$ is even.

Using the Hermitian metric to identify $A^*$ with $\bar A$, the pure component defines a holonomy-equivariant isomorphism
$$F:A\longrightarrow B\otimes\bar A.$$
The map $\Psi:B\otimes\bar A\to A$ is also holonomy equivariant.
Since $B$ is one-dimensional and $\Psi$ has no nonzero decomposable element in its kernel, $\Psi$ is injective and hence an isomorphism. Consequently,
$$\Psi\circ F:A\longrightarrow A$$
is a nonzero holonomy-equivariant endomorphism. By Lemma~\ref{lem:Schur}, it is a nonzero scalar multiple of the identity. Thus $\Psi$ is a scalar multiple of $F^{-1}$.

Since $F$ is holonomy equivariant, the positive Hermitian
endomorphism $F^*F$ commutes with the holonomy action on $A$.
Lemma~\ref{lem:Schur} therefore gives $F^*F=\lambda I_A$
for some $\lambda>0$. After rescaling $\Omega$, we may assume that
the matrix of $\Omega$ is unitary. Since $\Omega$ is also
skew-symmetric and nondegenerate, there is a unitary basis
$\{e_1,\ldots,e_m\}$ of $A$ in which
$$
\Omega=
\begin{pmatrix}
0&1&&\\
-1&0&&\\
&&\ddots&\\
&&&\begin{matrix}0&1\\-1&0\end{matrix}
\end{pmatrix}.
$$
In particular,
$$
\bar\Omega=\Omega,\qquad
\Omega^t=-\Omega,\qquad
\Omega\Omega^*=I,\qquad
\Omega^2=-I.
$$
After absorbing the remaining nonzero scalar factors into $\mu$ and
$\nu$, the block components needed below take the form
\begin{equation*}
T(e_i,e_j)=\mu\Omega_{ij}\xi,
\qquad
T(e_i,\bar\xi)
=\nu\sum_{p=1}^m\Omega_{ip}\bar e_p,
\qquad 1\le i,j\le m.
\end{equation*}
By complex conjugation and skew-symmetry,
$$
T(\bar e_i,\xi)
=\bar\nu\sum_{p=1}^m\bar\Omega_{ip}e_p,
\qquad
T(\xi,\bar e_i)
=-\bar\nu\sum_{p=1}^m\bar\Omega_{ip}e_p.
$$
All same-block mixed components and all pure components involving
both $A$ and $B$ vanish by the preceding discussion.

We now apply the Bianchi identity \eqref{eq:first-Bianchi-general}.
First take $(e_i,\bar\xi,\xi)$. Since $\nabla T=0$, the derivative
terms vanish. After projecting to $A$, the only curvature term is
$R(\bar\xi,\xi)e_i=-R(\xi,\bar\xi)e_i$.
On the torsion side, the only nonzero term is
$$
\begin{aligned}
T(T(e_i,\bar\xi),\xi)
&=
|\nu|^2\sum_{p,q}
\Omega_{ip}\bar\Omega_{pq}e_q=-|\nu|^2e_i,
\end{aligned}
$$
where we used $\Omega\bar\Omega=\Omega^2=-I$. Hence
$R(\xi,\bar\xi)e_i=|\nu|^2e_i$,
and therefore
\begin{equation}\label{eq:xi-A-curvature}
R_{\xi\bar\xi i\bar i}=|\nu|^2.
\end{equation}
Averaging over $i$ gives
$\rho_A(\xi,\bar\xi)=|\nu|^2$.
Since $B$ is a complex line and $|\xi|=1$,
$\rho_B(\xi,\bar\xi) = R_{\xi\bar\xi\xi\bar\xi} =H(\xi)=c$.
The relation $\rho_B=2\rho_A$ therefore gives
\begin{equation}\label{eq:c-nu}
c=2|\nu|^2.
\end{equation}

Next apply the Bianchi identity to $(e_i,\bar e_i,\xi)$ and project
to $B=\C\xi$. The only curvature term in $B$ is
$R(e_i,\bar e_i)\xi$. Since
$T(e_i,\bar e_i)=0$
and
$T(\xi,e_i)=0$,
while
$T(\bar e_i,\xi) = \bar\nu\sum_{p=1}^m\bar\Omega_{ip}e_p$,
we obtain
$$
\begin{aligned}
T(T(\bar e_i,\xi),e_i)
&=
\mu\bar\nu
\sum_p\bar\Omega_{ip}\Omega_{pi}\xi=-\mu\bar\nu\,\xi.
\end{aligned}
$$
Thus
\begin{equation*}
R_{i\bar i\xi\bar\xi}=-\mu\bar\nu.
\end{equation*}
Since $e_i$ and $\xi$ lie in distinct parallel blocks, the two mixed
output terms in \eqref{eq:HSC-polarization} vanish. Hence
$R_{i\bar i\xi\bar\xi} + R_{\xi\bar\xi i\bar i} =2c$.
Combining this identity with
\eqref{eq:xi-A-curvature} and \eqref{eq:c-nu}, we obtain
\begin{equation}\label{eq:symplectic-basic}
R_{\xi\bar\xi i\bar i}=|\nu|^2,
\qquad
c=2|\nu|^2,
\qquad
-\mu\bar\nu=\frac32c.
\end{equation}

For $1\le i,j\le m$, set
$A_{ij}=R_{i\bar i j\bar j}$.
For $i\neq j$, Lemma~\ref{lem:curvature-difference}, applied to
$(e_i,\bar e_j,e_j)$ and paired with $\bar e_i$, gives
$$
R_{i\bar j j\bar i}-A_{ji}
=
\langle\mathcal Q(e_i,\bar e_j,e_j),\bar e_i\rangle.
$$
The first two terms in $\mathcal Q(e_i,\bar e_j,e_j)$ vanish because
$T(A,\bar A)=0$. For the remaining term,
$$
T(e_j,e_i)=-\mu\Omega_{ij}\xi,
\qquad
T(\xi,\bar e_j)
=-\bar\nu\sum_p\bar\Omega_{jp}e_p.
$$
Therefore
$$
\langle
T(T(e_j,e_i),\bar e_j),\bar e_i
\rangle
=
-\mu\bar\nu|\Omega_{ij}|^2.
$$
Using \eqref{eq:symplectic-basic}, we obtain
$R_{i\bar j j\bar i}-A_{ji} = \frac32c|\Omega_{ij}|^2$.
Interchanging $i$ and $j$ gives
$R_{j\bar i i\bar j}-A_{ij} = \frac32c|\Omega_{ij}|^2$.
Applying \eqref{eq:HSC-polarization} to the orthonormal pair
$e_i,e_j$, we have
$A_{ij} +R_{j\bar i i\bar j} +R_{i\bar j j\bar i} +A_{ji} =2c$.
Substitution gives
\begin{equation}\label{eq:symplectic-pair}
A_{ij}+A_{ji}
=
c-\frac32c|\Omega_{ij}|^2.
\end{equation}

Since $A_{ii}=H(e_i)=c$,
$$\sum_{i,j}A_{ij} = mc+\sum_{i<j}(A_{ij}+A_{ji}).$$
Moreover, $\Omega$ is unitary and skew-symmetric, so
$$\sum_{i,j}|\Omega_{ij}|^2=m, \qquad \Omega_{ii}=0,$$
and therefore
$\sum_{i<j}|\Omega_{ij}|^2=m/2$.
Using \eqref{eq:symplectic-pair}, we obtain
$$\sum_{i,j}A_{ij}=mc+\binom m2c-\frac32c\frac m2
=\frac{m(2m-1)}4c.$$
We now compute the same sum using $\rho_B=2\rho_A$. Since $B$ is a complex line, \eqref{eq:symplectic-basic} gives $\rho_B(e_i,\bar e_i)=R_{i\bar i\xi\bar\xi}=\frac32c$, whereas $\rho_A(e_i,\bar e_i)=\frac1m\sum_jA_{ij}$. Hence $\frac32c=\frac2m\sum_jA_{ij}$, and therefore $\sum_jA_{ij}=\frac{3m}{4}c$. Summing over $i$ gives
$\sum_{i,j}A_{ij}=\frac{3m^2}{4}c$. Comparing the two expressions
for $\sum_{i,j}A_{ij}$ yields $\frac{m(2m-1)}4c=\frac{3m^2}{4}c$. Since $m>0$, this gives $(m+1)c=0$, contradicting $c\neq0$.

\smallskip
\noindent\emph{Case 3b: there is no nonzero pure component
$\Lambda^2A\to B$.}
In this case the doubling relation must arise from a nonzero mixed
component
$$B\otimes\bar A\longrightarrow A.$$
We first show that both $A$ and $B$ are clean. For $T(A,A)$, a
component with values in $B$ is absent by the assumption of Case 3b, while a component with values in $A$ would give $\rho_A=2\rho_A$, and hence $\rho_A=0$, which is impossible.
Similarly, a component of $T(B,B)$ with values in $A$ would give
$\rho_A=2\rho_B=4\rho_A$, again impossible since $\rho_A\neq0$, while a component with values in $B$ would give $\rho_B=2\rho_B$. Finally, every same-block mixed component would give a block with zero normalized Ricci form, which cannot occur since all coefficients $k_i$ are nonzero. Therefore
$$T(A,A)=T(A,\bar A)=T(B,B)=T(B,\bar B)=0,$$
and both $A$ and $B$ are clean.

Let
$m=\dim_\C A$
and
$q=\dim_\C B$.
By Lemma~\ref{lem:clean-block},
$$d_A=\frac{m+1}{2m}c, \qquad d_B=\frac{q+1}{2q}c.$$
Applying Lemma~\ref{lem:quadratic} to the relation $2\rho_A-\rho_B=0$ gives $4d_A+d_B=4c$. Substituting the above expressions for $d_A$ and $d_B$, we obtain $$4\frac{m+1}{2m}+\frac{q+1}{2q}=4,$$
or equivalently,
\begin{equation*}
\frac4m+\frac1q=3.
\end{equation*}
Since $m$ and $q$ are positive integers, $m=1$ is impossible. If
$m\ge3$, then
$$q=\frac{m}{3m-4}<1,$$
which is also impossible. Hence
$$m=2,\qquad q=1.$$
Choose a unit vector $\xi$ spanning the complex line $B$.

Let $x\in A$ be a unit vector. Apply the first Bianchi identity
\eqref{eq:first-Bianchi-general} to $(x,\bar x,\xi)$ and project to
$B$. Since $\nabla T=0$, all derivative terms vanish. Moreover,
$T(x,\bar x)=0$ because $A$ is clean, and
$T(\xi,x)=0$
because a nonzero pure component involving $A$ and $B$ would give,
by Lemma~\ref{lem:block-Ricci-relations}, a block Ricci form
$$\rho_A+\rho_B=3\rho_A,$$
which is not present in this two-block component. The remaining
quadratic torsion term contains a factor in $T(A,A)$ and therefore
also vanishes. Thus the $B$-component of the torsion side is zero.

Since curvature preserves the parallel blocks, among the curvature
terms only $R(x,\bar x)\xi$ has a $B$-component. Hence
$R(x,\bar x)\xi=0$,
and therefore
\begin{equation}\label{eq:two-block-cross-zero}
R_{x\bar x\xi\bar\xi}=0.
\end{equation}

Applying the polarization identity \eqref{eq:HSC-polarization} to
the orthonormal pair $x,\xi$, the two off-block terms vanish by block
preservation. Thus
$R_{x\bar x\xi\bar\xi} + R_{\xi\bar\xi x\bar x} =2c$.
By \eqref{eq:two-block-cross-zero},
$R_{\xi\bar\xi x\bar x}=2c$.
Since this holds for every unit vector $x\in A$, averaging over a
unitary basis of $A$ gives
$\rho_A(\xi,\bar\xi)=2c$.
On the other hand, $B$ is a complex line, so
$\rho_B(\xi,\bar\xi) = R_{\xi\bar\xi\xi\bar\xi} = H(\xi) = c$.
The relation $\rho_B=2\rho_A$ now gives
$c=4c$,
contradicting $c\neq0$.

Cases 1--3 cover all possibilities, and each leads to a contradiction. Hence no torsion-connected component can contain a
doubling relation when $c\neq0$. This proves the proposition.
\end{proof}

\section{Torsion components involving three distinct blocks}\label{sec:three-block}

Throughout this section, we continue to assume that
$\nabla T=0$ and $H\equiv c\neq0$. By the case analysis in Section~\ref{sec:repeated-components}, every nonzero block component with a
repeated underlying block index either forces a block with trivial
central character or produces a doubling relation. The first
possibility is impossible by Lemma~\ref{lem:trivial-central-character}, while the second is impossible
by Proposition~\ref{prop:doubling-exclusion}. Hence every nonzero
block component involves three pairwise distinct underlying block
indices.

In particular, every block in the torsion-connected component under
consideration is clean, since a nonzero block component of $T(V_i,V_i)$ or
$T(V_i,\bar V_i)$ would have a repeated underlying block index.

We shall use Lemmas~\ref{lem:block-Ricci-relations},~\ref{lem:quadratic},~\ref{lem:clean-block},~\ref{lem:null-pair}, and~\ref{lem:Segre-dimension}.
It remains to show that block components involving three pairwise distinct underlying block indices cannot occur.

\begin{proposition}\label{prop:three-block}
Under the standing assumptions of this section, no nonzero block
component of the complexified torsion involving three pairwise
distinct underlying block indices can occur.
\end{proposition}

\begin{proof}
Suppose that $T\neq0$, and fix a
torsion-connected component containing a nonzero block
component. We proceed in five steps.

\smallskip
\noindent\emph{Step 1: distinct nonzero block-component relations have the same support.}

For every clean block $V_i$, Lemma~\ref{lem:clean-block} gives
$d_i=\frac{n_i+1}{2n_i}c$.
Hence
$$
d_i-c=-\lambda_i c,
\qquad
\lambda_i=\frac{n_i-1}{2n_i}
\in\left[0,\frac12\right).
$$
By Lemma~\ref{lem:block-Ricci-relations}, every nonzero block
component involving three pairwise distinct blocks gives a relation
vector with exactly three nonzero coefficients, each equal to
$\pm1$, and with coefficient sum $\pm1$.

Let $a=(a_i)$ and $b=(b_i)$ be two such relation vectors. Applying
the polarized identity \eqref{eq:quadratic-polarized} to $a$ and
$b$, we obtain
$$
c\left(\sum_i a_i\right)\left(\sum_i b_i\right)
+\sum_i(d_i-c)a_ib_i=0.
$$
Using $d_i-c=-\lambda_i c$ and dividing by $c\neq0$, we get
\begin{equation}\label{eq:edge-support}
\left(\sum_i a_i\right)\left(\sum_i b_i\right)
=
\sum_{i\in\operatorname{supp}a\cap\operatorname{supp}b}
\lambda_i a_ib_i.
\end{equation}
The absolute value of the left-hand side is $1$. If
$\operatorname{supp}a\neq\operatorname{supp}b$, then, since both
supports contain exactly three indices, their intersection contains
at most two indices. Therefore
$$
\left|
\sum_{i\in\operatorname{supp}a\cap\operatorname{supp}b}
\lambda_i a_ib_i
\right|
\le
\sum_{i\in\operatorname{supp}a\cap\operatorname{supp}b}\lambda_i
<1,
$$
contradicting \eqref{eq:edge-support}. Thus all nonzero block
components in the chosen torsion-connected component have the same
three support blocks. Denote them by $A,B,C$.

\smallskip
\noindent\emph{Step 2: all nonzero block components give the same Ricci relation.}

Since each relation vector has three nonzero coefficients equal to
$\pm1$ and coefficient sum $\pm1$, up to multiplication by $-1$ the
only possible relations supported on $A,B,C$ are
$$
\rho_A+\rho_B-\rho_C=0,\qquad
\rho_A-\rho_B+\rho_C=0,\qquad
-\rho_A+\rho_B+\rho_C=0.
$$
Suppose that two different relations occur. Adding or subtracting
the corresponding equations gives $2\rho_A=0$, $2\rho_B=0$, or
$2\rho_C=0$. This is impossible. Indeed, for every clean block $V_i$, Lemma~\ref{lem:clean-block} gives
$$\rho_i|_{V_i}=d_i g_i, \qquad d_i=\frac{n_i+1}{2n_i}c\neq0,$$
so $\rho_i\neq0$.

Hence all nonzero block components in the chosen
torsion-connected component give the same Ricci relation. After
relabeling $A,B,C$ if necessary, we may assume
\begin{equation}\label{eq:ABC-relation}
\rho_C=\rho_A+\rho_B.
\end{equation}

\smallskip
\noindent\emph{Step 3: the three compatible block components are
nonzero and have no decomposable zeros.}
We first determine the block components compatible with
\eqref{eq:ABC-relation}. By Lemma~\ref{lem:block-Ricci-relations}, up to skew-symmetry and complex
conjugation, the only possibilities are
$$
P:A\otimes B\longrightarrow C,\qquad
Q:C\otimes\bar B\longrightarrow A,\qquad
S:C\otimes\bar A\longrightarrow B.
$$
The corresponding Ricci relations are
$$
\rho_C=\rho_A+\rho_B,\qquad
\rho_A=\rho_C-\rho_B,\qquad
\rho_B=\rho_C-\rho_A,
$$
respectively, and all three are equivalent to
\eqref{eq:ABC-relation}. Any other nonzero block component
involving these three blocks would give one of the other Ricci
relations considered in Step 2.

We first show that $P$ is nonzero. Suppose that $P=0$, and choose
nonzero vectors $x\in A$ and $y\in B$. Since both blocks are clean,
$$T(x,\bar x)=T(y,\bar y)=0.$$
There is no mixed component between $A$ and $B$ compatible with
\eqref{eq:ABC-relation}, and hence
$$T(x,\bar y)=T(y,\bar x)=0.$$
Moreover, $T(x,y)=P(x,y)=0$. Thus $(x,y)$ is a null pair in the sense of Lemma~\ref{lem:null-pair}, which gives $c=0$, a contradiction. Therefore $P\neq0$.

The same argument gives a stronger conclusion. If
$P(x,y)=0$
for some nonzero $x\in A$ and $y\in B$, then all the other torsion
terms appearing in Lemma~\ref{lem:null-pair} vanish as above. Hence that lemma again gives $c=0$. Therefore $P(x,y)\neq0$ whenever $x\neq0$ and $y\neq0$.

The same argument applies to the pair $(C,B)$. Up to skew-symmetry
and complex conjugation, the only mixed component compatible with
\eqref{eq:ABC-relation} is $Q:C\otimes\bar B\to A$. If $Q$ vanished on a nonzero decomposable tensor, the corresponding
vectors would form a null pair. Thus $Q$ is nonzero and has no
nonzero decomposable element in its kernel. Similarly, $S:C\otimes\bar A\to B$ is nonzero and has no nonzero decomposable element in its kernel.

\smallskip
\noindent\emph{Step 4: the Segre dimension estimate gives
$a=b=r=1$.}
Let
$$a=\dim_\C A,\qquad b=\dim_\C B,\qquad r=\dim_\C C.$$
Applying Lemma~\ref{lem:Segre-dimension} to $P$, $Q$, and $S$,
respectively, and noting that complex conjugation does not change
dimension, we obtain
$$r\ge a+b-1,\qquad a\ge r+b-1,\qquad b\ge r+a-1.$$
Adding these three inequalities gives $a+b+r\ge2(a+b+r)-3$, and hence $a+b+r\le3$. Since $a,b,r$ are positive integers, it follows that $a=b=r=1$.

\smallskip
\noindent\emph{Step 5: the line-block case contradicts the quadratic identity.}
Since $A$, $B$, and $C$ are clean complex lines, Lemma~\ref{lem:clean-block} gives
$$d_A=d_B=d_C=c.$$
With respect to the ordered triple $(A,B,C)$, relation
\eqref{eq:ABC-relation} has relation vector $(1,1,-1)$, whose coefficient sum is $1$. Applying Lemma~\ref{lem:quadratic},
we obtain
$$c+(d_A-c)+(d_B-c)+(d_C-c)=0.$$
Since $d_A=d_B=d_C=c$, this reduces to $c=0$, contradicting the standing assumption $c\neq0$.

Thus no nonzero block component involving three pairwise distinct
underlying block indices can occur. This proves the proposition.
\end{proof}

\section{Proof of the main theorem and applications}

We now combine the analysis of block components with a repeated underlying block index in Section~\ref{sec:repeated-components} with Proposition~\ref{prop:three-block} to complete the proof of the main theorem.

\begin{proof}[Proof of Theorem~\ref{thm:main}]
Suppose that $T_p\neq0$ at some point $p\in M$. With respect to the
irreducible holonomy decomposition
\eqref{eq:holonomy-decomposition-point}, choose a nonzero block
component of the complexified torsion at $p$.

If an underlying block index is repeated, the case analysis in
Section~\ref{sec:repeated-components} shows that the component either forces a block with trivial central character or produces a doubling relation. The first possibility is incompatible
with Lemma~\ref{lem:trivial-central-character} when $c\neq0$, while the second is impossible by Proposition~\ref{prop:doubling-exclusion}. If instead the three underlying block indices are pairwise distinct, the component is one of those treated in Proposition~\ref{prop:three-block}, and hence cannot occur when $c\neq0$.

It follows that every block component of the complexified
torsion at $p$ vanishes, contradicting $T_p\neq0$. Therefore,
$T=0$ on $M$, and hence $g$ is K\"ahler. Moreover, $H\equiv c$ is now the K\"ahler holomorphic sectional curvature of $g$. Therefore, $g$ is locally a complex space form.
\end{proof}

\begin{proof}[Proof of Corollary~\ref{cor:Chern}]
Let $\nabla=\nabla^c$ be the Chern connection. Its curvature is of
type $(1,1)$. Since $\nabla^cT^c=0$, Theorem~\ref{thm:partial-AS} gives $\nabla^cR^c=0$. Thus $\nabla^c$ is an
Ambrose--Singer connection.

If $c\neq0$, Theorem~\ref{thm:main} gives $T^c=0$; consequently,
$g$ is K\"ahler and locally a complex space form.

Now assume that $c=0$. On each connected component of $M$, the
Ni--Zheng structure theorem for complete locally Chern homogeneous Hermitian manifolds \cite{NiZhengCAS} applies, since both the Chern torsion and curvature are $\nabla^c$-parallel. The universal cover therefore splits into a Chern-flat complex Lie group and Hermitian symmetric factors. Since $H^c\equiv0$, each Hermitian symmetric factor has vanishing curvature by the K\"ahler polarization identity. Hence all factors are Chern flat, and consequently $R^c=0$ on $M$.
\end{proof}

\begin{proof}[Proof of Corollary~\ref{cor:BAS}]
Let $\nabla=\nabla^b$. By hypothesis, $\nabla^bT^b=0$. The curvature characterization of Bismut torsion-parallel metrics proved by Zhao and Zheng shows that $R^b$ is of type $(1,1)$; see \cite{ZhaoZBTP}. Theorem~\ref{thm:partial-AS} therefore gives
$\nabla^bR^b=0$, so $\nabla^b$ is an Ambrose--Singer connection.
If the Bismut holomorphic sectional curvature is a nonzero constant, Theorem~\ref{thm:main} gives $T^b=0$. Consequently, $g$ is K\"ahler and locally a complex space form.
\end{proof}

\noindent\textbf{Generative AI disclosure.}
During the development of this work, the authors used ChatGPT
5.6 Sol to assist with exploratory computations, possible proof directions, and editorial polishing. The authors checked the mathematical proofs and computations, and take full responsibility for the contents of the paper.

\vspace{0.3cm}

\noindent\textbf{Declaration of competing interests.}
The authors declare that they have no competing interests.

\end{document}